\documentclass[10pt]{article}
\usepackage{indentfirst, latexsym, amsmath,bm,amssymb}
\usepackage{tikz}
\usepackage{color,soul}
\usepackage{amsmath}
\usepackage{amssymb}
\usepackage{mathrsfs}
\usepackage{mathtools}
\usepackage{stmaryrd}

\begin{document}

\newcommand{\Ecc}{\operatorname{Ecc}}
\newcommand{\ecc}{\operatorname{ecc}}
\newcommand{\Uni}{\operatorname{Uni}}
\newcommand{\uni}{\operatorname{uni}}
\newcommand{\diam}{\operatorname{diam}}
\newcommand{\rad}{\operatorname{rad}}
\newcommand{\degree}{\operatorname{deg}}
\renewcommand{\deg}{\overline{d}} 

\newtheorem{theo}{Theorem}[section]
\newtheorem{theorem}{Theorem}[section]
\newtheorem{definition}[theorem]{Definition}
\newtheorem{prop}[theo]{Proposition}
\newtheorem{lemma}[theo]{Lemma}
\newtheorem{claim}[theo]{Claim}
\newtheorem{cor}[theo]{Corollary}
\newtheorem{question}{Question}[section]
\newtheorem{remark}[theo]{Remark}
\newtheorem{defi}{Definition}
\newtheorem{conj}[theo]{Conjecture}
\newtheorem{ob}[theo]{Observation}
\def\qedsymbol{\ensuremath{\scriptstyle\blacksquare}}
\newenvironment{proof}{\noindent {\bf
Proof.}}{\rule{3mm}{3mm}\par\medskip}

\renewcommand{\deg}{\overline{d}}
\newcommand{\JEC}{{\it Europ. J. Combinatorics},  }
\newcommand{\JCTB}{{\it J. Combin. Theory Ser. B.}, }
\newcommand{\JCT}{{\it J. Combin. Theory}, }
\newcommand{\JGT}{{\it J. Graph Theory}, }
\newcommand{\ComHung}{{\it Combinatorica}, }
\newcommand{\DM}{{\it Discrete Math.}, }
\newcommand{\ARS}{{\it Ars Combin.}, }
\newcommand{\SIAMDM}{{\it SIAM J. Discrete Math.}, }
\newcommand{\SIAMADM}{{\it SIAM J. Algebraic Discrete Methods}, }
\newcommand{\SIAMC}{{\it SIAM J. Comput.}, }
\newcommand{\ConAMS}{{\it Contemp. Math. AMS}, }
\newcommand{\TransAMS}{{\it Trans. Amer. Math. Soc.}, }
\newcommand{\AnDM}{{\it Ann. Discrete Math.}, }
\newcommand{\NBS}{{\it J. Res. Nat. Bur. Standards} {\rm B}, }
\newcommand{\ConNum}{{\it Congr. Numer.}, }
\newcommand{\CJM}{{\it Canad. J. Math.}, }
\newcommand{\JLMS}{{\it J. London Math. Soc.}, }
\newcommand{\PLMS}{{\it Proc. London Math. Soc.}, }
\newcommand{\PAMS}{{\it Proc. Amer. Math. Soc.}, }
\newcommand{\JCMCC}{{\it J. Combin. Math. Combin. Comput.}, }
\newcommand{\GC}{{\it Graphs Combin.}, }
\title{Variations of the eccentricity and their properties in trees\thanks{
This work is supported by the National Natural Science Foundation of China (Nos.11601208,11531001 and 11601337),  the Joint NSFC-ISF Research Program (jointly funded by the National Natural Science Foundation of China and the Israel Science Foundation (No.11561141001)) and  the Montenegrin-Chinese
Science and Technology Cooperation Project (No.3-12).
  }}
\author{ Ya-Hong  Chen$^{1}$,  Hua Wang$^{2,3}$, Xiao-Dong Zhang$^4$\thanks{Corresponding  author ({\it E-mail address: xiaodong@sjtu.edu.cn})}
\\
{\small $^1$Department of Mathematics},
{\small Lishui University} \\
{\small  Lishui, Zhejiang 323000, PR China}\\
 {\small $^2$ College of Software, Nankai University}\\
{\small Tianjin 300071, P.R. China} \\
{\small $^3$Department of Mathematical Sciences,}
{\small Georgia Southern University }\\
{\small Statesboro, GA 30460 USA}\\
{\small $^4$ School of Mathematical Sciences, MOE-LSC, SHL-MAC,}
{\small Shanghai Jiao Tong University} \\
{\small  800 Dongchuan road, Shanghai, 200240,  P.R. China}
}

\date{}
\maketitle
 \begin{abstract}
Motivated from the study of eccentricity, center, and sum of eccentricities in graphs and trees, we introduce several new distance-based global and local functions based on the smallest distance from a vertex to some leaf (called the ``uniformity'' at that vertex). Some natural extremal problems on trees are considered. Then the middle parts of a tree is discussed and compared with the well-known center of a tree. The values of the global functions are also compared with the sum of eccentricities and some sharp bounds are established. Last but not the least, we show that the difference between the eccentricity and the uniformity, when considered as a local function, behaves in a very similar way as the eccentricity itself.
 \end{abstract}


\vskip 0.5cm

\section{Introduction}

Graph invariants used as topological indices have been extensively studied in mathematics and many related fields. Among numerous graph invariants many are defined on the distances between vertices. One well-known such ``distance-based'' graph invariant is the {\it Wiener index}, defined as the sum of distances between all pairs of vertices in a graph $G$ \cite{wiener1947, wiener1947'}. Denoted by $W(G)$, it can also be represented as
$$ W(G) = \frac12 \sum_{v\in V(G)} d(v) $$
where
$$ d(v) = \sum_{u \in V(G)} d_G(u,v)$$
is often called the {\it distance function} at $v$, considered as the local version of the global function $W(G)$. Here $d_G(u,v)$ is the distance between two vertices $u$ and $v$ in a graph $G$. We often simply write $d(u,v)$ when it is clear what the underlying graph is.

Analogous to $d(v)$, another well studied distance-based concept is called the {\it eccentricity}, defined as
$$ \ecc_G(v) = \max_{u \in V(G)} d(u,v) . $$
Again we often write $ecc(v)$ (instead of $\ecc_G(v)$) when there is no confusion. Treating $\ecc(v)$ as a local function, the global function is naturally defined as the sum of eccentricities
$$\Ecc(G) = \sum_{v \in V(G)} \ecc(v)$$ and is introduced in \cite{heather}.

Examining the behaviors of various graph invariants in trees has been of interest. The sum of eccentricities in trees was extensively studied in \cite{heather}. We will also focus on trees in this paper.

\subsection{Variations of $\ecc(v)$ and $\Ecc(T)$}

We start with introducing some natural variations that we will study throughout this paper. First note that the eccentricity of a vertex $v$ must be obtained by the distance between $v$ and a leaf (since it is the largest distance from $v$), and consequently one could write
$$ \ecc(v) = \max_{u \in L(T)} d(u,v)  $$
where $L(T)$ is the set of leaves of $T$. It is then natural to consider, for any vertex $v$ in a tree $T$, the function
$$ \uni_T(v) := \min_{u \in L(T)} d(u,v) . $$
We will call $\uni_T(v)$ (or simply $\uni(v)$) the {\it uniformity} of $v$ in $T$. Note that the uniformity of a leaf is, by definition, zero.

Consequently, the sum of uniformities of a tree $T$ is denoted by
$$ \Uni(T) := \sum_{v \in V(T)} \uni(v) . $$
We denote the difference between the eccentricity and uniformity at a vertex by
$$\delta_T(v):= \ecc_T(v) - \uni_T(v) $$
and its sum by
$$ \Delta(T) = \sum_{v \in V(T)} \delta_T(v) . $$

Lastly, note that $\Ecc(T)$ is the sum of the largest distances from vertices. The natural analogue is the largest sum of distances from vertices, i.e.
$$ LD(T):=\max_{v \in V(T)} d(v) . $$

\subsection{Extremal problems for global functions}

The so-called extremal problems ask for the extremal structures that maximize or minimize a graph invariant within a collection of graphs. For instance, the extremal trees that minimize or maximize the Wiener index in various classes of trees have been studied. One may see \cite{dobrynin2001,gutman-k-1,2016knor,xu2014} for surveys of such studies on the Wiener index and distance-based indices in general.  The extremal problems with respect to the sum of eccentricities have been studied in \cite{heather}. We will, in Section~\ref{sec:ex}, consider such problems with respect to $\Uni(T)$ and $LD(T)$ among trees and trees with a given number of internal vertices.

\subsection{Middle parts of a tree}

The ``middle part'' of a tree is usually defined as the collection of vertices that maximize or minimize a certain local graph invariant in a tree. The collection of vertices that minimize $d(v)$ is called the {\it centroid} $CT(T)$ of $T$ and the collection of vertices that minimize $\ecc(v)$ is called the {\it center} $C(T)$ of $T$ \cite{1974adam}. It is well known that these two middle parts share the common property that  each contains one or two adjacent vertices. It is also known that they do not need to be the same. The study of different middle parts and their relations has received some attention in the recent years \cite{heather1, hua2014, hua2015}.

From $\uni(v)$ a natural middle part can be defined as the collection of vertices that maximize $\uni(v)$, denoted by $C_{\uni}(T)$. In Section~\ref{sec:mid} we briefly examine the property of $C_{\uni}(T)$ and compare it with $C(T)$. In addition, as the vertices in $C(T)$ achieves the {\it radius}
\begin{equation}
\label{eq:rad}
r(T):=\min_{v \in V(T)} \left( \max_{u\in L(T)} d(v, u) \right)
\end{equation}
while vertices in $C_{\uni}(T)$ achieves
\begin{equation}
\label{eq:norm}
r'(T):=\max_{v \in V(T)} \left( \min_{u\in L(T)} d(v, u) \right) .
\end{equation}
It is also a natural question to compare their values. We will see in Section~\ref{sec:mid} that the radius is always larger.

\subsection{Comparison between different concepts}

As our study is motivated from the eccentricities, it makes sense to compare $\Uni(T)$ and $LD(T)$ with $\Ecc(T)$. This is done in Section~\ref{sec:com}. Also note that
$$ \Delta(T) = \Ecc(T) - \Uni(T) $$
and
$$ \delta(v) = \ecc(v) - \uni(v) . $$
In fact, when $\delta(v)$ is considered as a local function, it behaves in a very similar way as $\ecc(v)$. We will discuss related observations in Section~\ref{sec:delta}.

\section{Extremal trees with respect to $\Uni(T)$ and $LD(T)$}
\label{sec:ex}

We first note the following simple observation.

\begin{prop}
For a tree $T$ of order $n$,
$$LD(S_n)\leq LD(T)\leq LD(P_n) $$
where $S_n$ and $P_n$ are the star and path on $n$ vertices, respectively.
\end{prop}

\begin{proof}
By definition we have
$$ d(v) \geq 1 + (n-2)\cdot 2 $$
for any leaf $v$, as $v$ has a unique neighbor and is at distance at least 2 from any other vertex. Thus
$$ LD(T) = \max_{v\in V(T)} d(v) \geq 1 + (n-2)\cdot 2 = LD(S_n) $$
with equality if and only if $T$ is a star.

On the other hand, given a vertex $v \in V(T)$, for every vertex $u$ such that $d(u,v)=k\geq 2$ there exists at least one vertex $w$ such that $d(w,v)=k-1$, hence
$$ d(v) \leq 1+2 + \ldots + (n-1) = \frac{n(n-1)}{2} $$
with equality if and only if $v$ is one end of a path. Then
$$ LD(T) = \max_{v\in V(T)} d(v) \leq  \frac{n(n-1)}{2} = LD(P_n) . $$
\end{proof}

Next, for $\Uni(T)$, we will show that the path $P_n$ also maximizes it among trees of order $n$.

\begin{theorem}
\label{thm:path}
Among all trees of order $n$, $\Uni(T)$ is maximized by the path $P_n$. More specifically,
\begin{equation*}
\Uni(T)\leq\left\{\begin{array}{ll}
\frac{n^2-2n+1}{4}, & {\rm if}\  n ~~is~~ odd; \\
\frac{n^2-2n}{4}, & {\rm if} \ n ~~is ~~even;\end{array}\right.
\end{equation*}
 with equality if and only if $T\cong P_n$.
\end{theorem}
\begin{proof}
We will focus on the case of even $n$. The odd case can be handled in exactly the same way.

Let $T$ be such a tree of order $n$ with maximum $\Uni(T)$, and suppose (for contradiction) that $T$ is not a path. Let $P:=v_0v_1\ldots v_d$ be a longest path in $T$ and let $T_i$ denote the component containing $v_i$ in $T-E(P)$ for $i=1,2,\ldots,d-1$. Let $t$ be the smallest subscript such that $|V(T_t)|\geq2$.

We consider the tree $T'$ obtained from $T$ by ``moving'' $T_t$ from $v_t$ to $v_0$. That is, we remove all edges between $v_t$ and its neighbors in $T_t$ and connect these neighbors to $v_0$ (Figure~\ref{fig:T}).

\begin{figure}[htbp]
\centering
    \begin{tikzpicture}[scale=1]
        \node[fill=black,circle,inner sep=1pt] (t1) at (0,0) {};
        \node[fill=black,circle,inner sep=1pt] (t2) at (1,0) {};
        \node[fill=black,circle,inner sep=1pt] (t3) at (2,0) {};
        \node[fill=black,circle,inner sep=1pt] (t4) at (3,0) {};

        \node[fill=black,circle,inner sep=1pt] (t5) at (4,0) {};

\node[fill=black,circle,inner sep=1pt] (t6) at (7,0) {};
        \node[fill=black,circle,inner sep=1pt] (t7) at (8,0) {};
        \node[fill=black,circle,inner sep=1pt] (t8) at (9,0) {};
        \node[fill=black,circle,inner sep=1pt] (t9) at (10,0) {};

        \node[fill=black,circle,inner sep=1pt] (t10) at (11,0) {};

        \draw (t1)--(t2);
       \draw (t4)--(t5);
        \draw [dashed](t2)--(t3);
        \draw [dashed] (t3)--(t4);

\draw (t6)--(t7);
       \draw (t9)--(t10);
        \draw [dashed](t8)--(t7);
        \draw [dashed] (t9)--(t8);

        \draw (t4)--(2.7, -.8)--(3.3, -.8)--cycle;
        \draw (t3)--(1.7, -.8)--(2.3, -.8)--cycle;

\draw (t6)--(6.2, .3)--(6.2, -.3)--cycle;
        \draw (t9)--(9.7, -.8)--(10.3, -.8)--cycle;

         \node at (0,.2) {$v_{0}$};
         \node at (2,.2) {$v_{t}$};
        \node at (3,.2) {$v_{d-1}$};
        \node at (1,.2) {$v_{1}$};
        \node at (4,.2) {$v_{d}$};
\node at (2,-.55) {$T_{t}$};

\node at (6.4,0) {$T_{t}$};
\node at (2,-1.5) {$T$};
\node at (9,-1.5) {$T'$};

\node at (7,.2) {$v_{0}$};
         \node at (9,.2) {$v_{t}$};
        \node at (10,.2) {$v_{d-1}$};
        \node at (8,.2) {$v_{1}$};
        \node at (11,.2) {$v_{d}$};


        \end{tikzpicture}
\caption{The trees $T$ and $T'$.}\label{fig:T}
\end{figure}
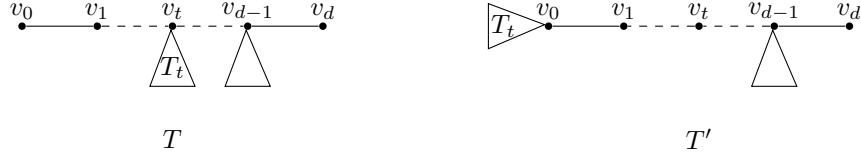

 Furthermore, let
$$x= \Uni_{T_t}(v_t) =\min_{w \in L(T_t)} d(v_t, w) .$$
It is easy to see that $x\leq t \leq \lfloor \frac{d}{2} \rfloor $ as $P$ was defined to be a longest path.

We will consider, from $T$ to $T'$, the possible change in $\uni(v)$ for every vertex.

\begin{itemize}
\item First, for any vertex $v\in V(T_t)\setminus \{v_t\}$, if $\uni_T(v)$ was obtained with a leaf in $T_t$ then $\uni_{T'}(v) = \uni_T(v)$. Otherwise, if $\uni_T(v)$ was obtained with a leaf outside of $T_t$, then $\uni_{T'}(v) \geq \uni_T(v)$.
\item For any vertex $v \in V(T)\setminus \left( \{ v_0, v_1, \ldots, v_{t-1} \} \cup (V(T_t)\setminus\{v_t\}) \right)$, since the leaves of $T_t$ are moved further away from $v$, we must have $\uni_{T'}(v) \geq \uni_T(v)$.
\item For any vertex $v \in  \{ v_0, v_1, \ldots, v_{t-1} \}$:
\begin{itemize}
\item If $\uni_T(v)$ was obtained by $v$ and $v_0$, then $\uni_{T'}(v) > \uni_T(v)$ as $v_0$ is no longer a leaf in $T'$.
\item If $\uni_T(v)$ was obtained  by $v$ and a leaf not in $T_t$, then by the first case the uniformity can not decrease. In particular, we have $\uni_{T'}(v_0) > \uni_T(v_0)$.
\item If $\uni_T(v)$ was obtained  by $v$ and a leaf $w$ in $T_t$, then $v=v_i$ must be closer to $v_t$ than to $v_0$ in $T$, and consequently $\uni_{T'}(v) = \min\{i+x,d(v,w')\} $ for some leaf $w'$ not in $T_t$.
\begin{itemize}
\item if $\uni_{T'}(v) = i+ x$, then $\uni_{T'}(v) =i+x \geq (t-i)+x = \uni_T(v)$ as $v=v_i$ is closer to $v_t$ than to $v_0$;
\item if $\uni_{T'}(v) = d(v,w') $, then by the definition of $\uni_T(v)$ we have $d(v,w') \geq d(v, w) =(t-i) +x$, then  $\uni_{T'}(v) =d(v,w')\geq (t-i)+x = \uni_T(v)$.
\end{itemize}
\end{itemize}
\end{itemize}

Consequently we have $\Uni(T') > \Uni(T)$, a contradiction.

Thus $\Uni(T)$ is maximized by a path, it is easy to compute the exact value depending on the parity of $n$.
\end{proof}

On the other hand, it is easy to see that the uniformity is at least 1 for any internal vertex and as the star has only one internal vertex $v$ with $\uni(v)=1$, we have the following observation.

\begin{prop}
For any tree $T$ of order $n$ we have
$$ \Uni(T) \geq 1 $$
with equality if and only if $T \cong S_n$.
\end{prop}

Since $\Uni(T)$ largely depends on the number of internal vertices, it makes sense to consider the extremal problem among trees (of order $n$) with a given number of internal vertices (or equivalently, with a given number of leaves). Through analogous argument as the proof of Theorem~\ref{thm:path}, we have the following.

\begin{prop}
Among trees of order $n$ with $k$ internal vertices, $\Uni(T)$ is maximized by a tree (not unique) obtained through attaching a total of $n-k$ pendant edges to the two ends of a path $P_k$ with at least one pendant edge at each end. See Figure~\ref{fig:dumb1}. More specifically, we have
\begin{equation*}
\Uni(T)\leq\left\{\begin{array}{ll}
\frac{k^2+2k+1}{4}, & {\rm if}\  k ~~is~~ odd; \\
\frac{k^2+2k}{4}, & {\rm if} \ k ~~is ~~even;\end{array}\right.
\end{equation*}
\end{prop}

\begin{figure}[!ht]
	\centering

		\begin{tikzpicture}
		  \node[fill=black,circle,inner sep=1pt] (t1) at (-4,0) {};
        \node[fill=black,circle,inner sep=1pt] (t2) at (-3,0) {};
		\node[fill=black,circle,inner sep=1pt] (t3) at (-1,0) {};
        \node[fill=black,circle,inner sep=1pt] (t4) at (0,0) {};
        \node[fill=black,circle,inner sep=1pt] (t5) at (1,1) {};
        \node[fill=black,circle,inner sep=1pt] (t6) at (1,0.4) {};
         \node[fill=black,circle,inner sep=1pt] (t8) at (1,0) {};
        \node[fill=black,circle,inner sep=1pt] (t7) at (1,-1) {};
        \node[fill=black,circle,inner sep=1pt] (u1) at (-5,1) {};
        \node[fill=black,circle,inner sep=1pt] (u2) at (-5,0) {};
        \node[fill=black,circle,inner sep=1pt] (u3) at (-5,-1) {};
       \draw (t1)--(t2) (t3)--(t4) (t5)--(t4);
       \draw  [dashed] (t3)--(t2) ;
       \draw (t4)--(t7) (t4)--(t6)(t4)--(t8);
       \draw (t1)--(u1) (t1)--(u2)(t1)--(u3);
       \node at (-4,0.2) {$v_{1}$};
        \node at (-3,0.2) {$v_{2}$};
        \node at (-1,.2) {$v_{k-1}$};
\node at (0,.2) {$v_{k}$};
        \node [rotate = 90] at (1,-.4) {$\cdots$};
         \node [rotate = 90] at (-5,-.4) {$\cdots$};
		\end{tikzpicture}
\caption{An extremal tree with maximum $\Uni(T)$.}\label{fig:dumb1}	
\end{figure}
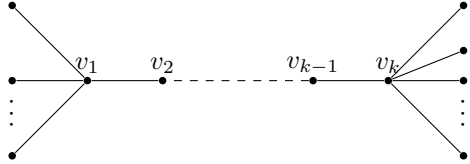

To minimize $\Uni(T)$ the conclusion depends on the number of internal vertices (compared with the total order of the tree). Recall that a {\it starlike tree} $S(l_1, l_2, \ldots , l_m)$ is the tree with exactly one vertex of degree $\geq 3$ formed by identifying the ends of $m$ paths of length $l_1,l_2,\ldots,l_m$, respectively. See Figure~\ref{fig:starlike_tree_ex} for an example.

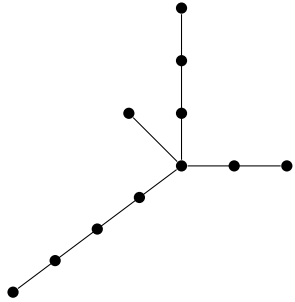
\begin{figure}[htbp]
\centering
    \begin{tikzpicture}[scale=0.7]
        \node[fill=black,circle,inner sep=1.5pt] (t1) at (0,0) {};
        \node[fill=black,circle,inner sep=1.5pt] (t2) at (1,0) {};
        \node[fill=black,circle,inner sep=1.5pt] (t3) at (2,0) {};
        \node[fill=black,circle,inner sep=1.5pt] (t4) at (-0.8,-0.6) {};
        \node[fill=black,circle,inner sep=1.5pt] (t5) at (-1.6,-1.2) {};
        \node[fill=black,circle,inner sep=1.5pt] (t6) at (-2.4,-1.8) {};
        \node[fill=black,circle,inner sep=1.5pt] (t7) at (-3.2,-2.4) {};
        \node[fill=black,circle,inner sep=1.5pt] (t8) at (0,1) {};
        \node[fill=black,circle,inner sep=1.5pt] (t9) at (0,2) {};
        \node[fill=black,circle,inner sep=1.5pt] (t10) at (0,3) {};
        \node[fill=black,circle,inner sep=1.5pt] (t11) at (-1,1) {};

\draw (t1)--(t3);
\draw (t1)--(t7);
\draw (t1)--(t10);
\draw (t1)--(t11);

        \end{tikzpicture}
\caption{An example of a starlike tree.}\label{fig:starlike_tree_ex}
\end{figure}

\begin{prop}
Let $T$ be a tree of order $n$ with $k$ internal vertices:
\begin{itemize}
\item If $k \leq \lfloor \frac{n}{2} \rfloor$, then
$$ \Uni(T) \geq k $$
with equality if and only if every internal vertex is adjacent to some leaf. Such an extremal tree is obviously not unique.
\item If $k > \lfloor \frac{n}{2} \rfloor$, then $\Uni(T)$ is minimized by a starlike tree $S(l_1, l_2, \ldots, l_{n-k})$ where $|l_i-l_j|\leq 1$ for any $1\leq i, j \leq n-k$.
\end{itemize}
\end{prop}

\begin{proof}
The first case is trivial, as $\uni(v)\geq 1$ for each of the $k$ internal vertices and this can be achieved when there are more leaves than internal vertices.

In the second case, we have $k > n-k$, the number of leaves. Then there are at most $n-k$ internal vertices with $\uni(v)=1$, then at most $n-k$ internal vertices with $\uni(v)=2$, etc. It is easy to see that the described starlike tree (again, not unique) achieves this.
\end{proof}

\section{Middle parts of a tree}
\label{sec:mid}

First we note that there is no obvious connection between the center $C(T)$ and $C_{\uni}(T)$. As one can see from Figure~\ref{fig.example3}, $C_{\uni}(T)$ may not induce a connected subgraph (recall that $C(T)$ contains one or two adjacent vertices, hence always connected). In this particular case we also have the two middle parts being disjoint from each other.

\begin{figure}[htbp]
\centering
    \begin{tikzpicture}[scale=1]
        \node[fill=black,circle,inner sep=1pt] (t1) at (0,0) {};
        \node[fill=black,circle,inner sep=1pt] (t2) at (-1,0) {};
        \node[fill=black,circle,inner sep=1pt] (t3) at (-2,0) {};
        \node[fill=black,circle,inner sep=1pt] (t4) at (1,0) {};
        \node[fill=black,circle,inner sep=1pt] (t5) at (2,0) {};
        \node[fill=black,circle,inner sep=1pt] (t6) at (3,0) {};
        \node[fill=black,circle,inner sep=1pt] (t7) at (4,0) {};
        \node[fill=black,circle,inner sep=1pt] (t8) at (1,-1) {};

        \draw (t1)--(t2);
        \draw (t2)--(t3);
\draw (t1)--(t4);
        \draw (t4)--(t5);
        \draw (t5)--(t6);

        \draw (t6)--(t7);

        \draw (t4)--(t8);

\node at (0,.2) {$u$};
\node at (1,.2) {$v$};
\node at (2,.2) {$w$};

        \end{tikzpicture}
\caption{A tree $T$ with $C(T)=\{v\}$ and $C_{\uni}(T)=\{u,w\}$.}\label{fig.example3}
\end{figure}
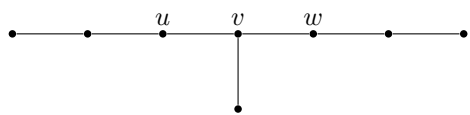

Furthermore, $C_{\uni}(T)$ may contain many vertices, as in a binary caterpillar shown in Figure~\ref{fig.example}, all internal vertices are in $C_{\uni}(T)$.

\begin{figure}[htbp]
\centering
    \begin{tikzpicture}[scale=1]
        \node[fill=black,circle,inner sep=1pt] (t1) at (0,0) {};
        \node[fill=black,circle,inner sep=1pt] (t2) at (-1,0) {};
        \node[fill=black,circle,inner sep=1pt] (t3) at (-2,0) {};
        \node[fill=black,circle,inner sep=1pt] (t4) at (1,0) {};
        \node[fill=black,circle,inner sep=1pt] (t5) at (2,0) {};
        \node[fill=black,circle,inner sep=1pt] (t6) at (3,0) {};
        \node[fill=black,circle,inner sep=1pt] (t7) at (-1,1) {};
        \node[fill=black,circle,inner sep=1pt] (t8) at (0,1) {};
        \node[fill=black,circle,inner sep=1pt] (t9) at (1,1) {};
        \node[fill=black,circle,inner sep=1pt] (t10) at (2,1) {};

        \draw (t1)--(t2);
        \draw (t2)--(t3);

        \draw (t4)--(t5);
        \draw (t5)--(t6);

        \draw (t2)--(t7);
        \draw (t1)--(t8);
        \draw (t4)--(t9);
        \draw (t5)--(t10);

 \node at (0.5,0) {$\ldots$};

        \end{tikzpicture}
\caption{A binary caterpillar.}\label{fig.example}
\end{figure}
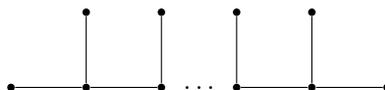

Also note that vertices of $C(T)$ and $C_{\uni}(T)$ need not be adjacent, as can be seen by an example similar to Figure~\ref{fig.example3} but with a longer diameter. Furthermore, $C(T)$ does not need to lie ``between'' vertices of $C_{\uni}(T)$ as shown in Figure~\ref{fig:ex'}.

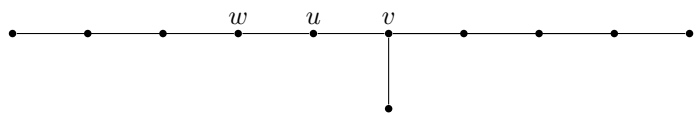
\begin{figure}[htbp]
\centering
    \begin{tikzpicture}[scale=1]
\node[fill=black,circle,inner sep=1pt] (t0) at (-3,0) {};
        \node[fill=black,circle,inner sep=1pt] (t1) at (0,0) {};
        \node[fill=black,circle,inner sep=1pt] (t2) at (-1,0) {};
        \node[fill=black,circle,inner sep=1pt] (t3) at (-2,0) {};
        \node[fill=black,circle,inner sep=1pt] (t4) at (1,0) {};
        \node[fill=black,circle,inner sep=1pt] (t5) at (2,0) {};
        \node[fill=black,circle,inner sep=1pt] (t6) at (3,0) {};
        \node[fill=black,circle,inner sep=1pt] (t7) at (4,0) {};
        \node[fill=black,circle,inner sep=1pt] (t8) at (1,-1) {};
        \node[fill=black,circle,inner sep=1pt] (t9) at (5,0) {};
        \node[fill=black,circle,inner sep=1pt] (t10) at (-4,0) {};

        \draw (t1)--(t2);
        \draw (t2)--(t10);
\draw (t1)--(t4);
        \draw (t4)--(t5);
        \draw (t5)--(t6);

        \draw (t6)--(t7);

        \draw (t8)--(t4)--(t9);

\node at (0,.2) {$u$};
\node at (1,.2) {$v$};
\node at (-1,.2) {$w$};

        \end{tikzpicture}
\caption{A tree $T$ with $C(T)=\{u, v\}$ and $C_{\uni}(T)=\{w\}$.}\label{fig:ex'}
\end{figure}

Related to the middle parts, recall that we defined $r(T)$ and $r'(T)$ in \eqref{eq:rad} and \eqref{eq:norm}. It is easy to show that the radius $r(T)$ is always larger.

\begin{prop}
For any tree $T$, we must have
$$  r(T) \geq r'(T). $$
\end{prop}
\begin{proof}
Let $v_0\in C_{\uni}(T)$ and $v_1 \in C(T)$, then
$$r(T)=\max_{u\in L(T)}d(u,v_1) \hbox{ and } r'(T)=\min_{u\in L(T)}d(u,v_0). $$

Now pick a leaf $w$ that is closer to $v_0$ than $v_1$, then
$$ r(T)=\max_{u\in L(T)}d(u,v_1)\geq d(w,v_1)\geq d(w,v_0)\geq \min_{u\in L(T)}d(u,v_0) = r'(T). $$
\end{proof}

\section{Comparing $\Uni(T)$ and $LD(T)$ with $\Ecc(T)$}
\label{sec:com}

In this section we study the difference between $\Ecc(T)$ and the new distance-based graph invariants.

\subsection{The difference between $\Ecc(T)$ and $LD(T)$}

 First of all it is easy to see, by definition, that $\Ecc(T)$ is at least as large as $LD(T)$ as
$$ LD(T) = \max_{v\in V(T)} d(v) = d(v_0) = \sum_{u \in V(T)} d(v_0, u) \leq \sum_{u \in V(T)} \ecc(u) = \Ecc(T) $$
where we assume $LD(T)$ to be obtained at the vertex $v_0$.

The next observation states that $\Ecc(T)$ is at least two more than $LD(T)$.

\begin{prop}\label{prop:41'}
For any tree $T$ of order at least 3,
$$ \Ecc(T) \geq LD(T) + 2 $$
with equality if and only if $T$ is a star.
\end{prop}
\begin{proof}
Again let $v_0$ be the vertex where $LD(T)=\max_{v\in V(T)} d(v)$ is obtained. Note that $v_0$ has to be a leaf vertex. Then
\begin{align*}
 LD(T) = d(v_0) & =  \sum_{u \in V(T)} d(u,v_0) = d(v_0,v_0) +  \sum_{u\neq v_0} d(u,v_0) \\
& \leq \sum_{u\neq v_0} \ecc(u) = \Ecc(T) - \ecc(v_0) \leq \Ecc(T) - 2 .
\end{align*}
It is easy to see that equality holds in both inequalities if and only if $v_0$ is a leaf of a star.
\end{proof}

Proposition~\ref{prop:41'} states that the difference $\Ecc(T) - LD(T)$ is at least 2 in trees. It is natural to expect the path to maximize this difference. We have, however, not yet been able to find a proof.

\begin{question}\label{prop2}
Is it true that we always have
$$ \Ecc(T)-LD(T)\le  \Ecc(P_n) - LD (P_n) $$
for any tree $T$ of order $n$?
\end{question}

\subsection{The difference $\Delta(T) = \Ecc(T) - \Uni(T)$}

Similarly, it makes sense to consider the extremal values of $\Delta(T) = \Ecc(T) - \Uni(T)$. By definition one immediately has $\Delta(T) \geq 0$ for any $T$. We now show that the minimum $\Delta(T)$ is achieved by the star.

\begin{theorem}
Among all trees of order $n$, $\Delta(T) \geq 2(n-1) = \Delta(S_n)$.
\end{theorem}

\begin{proof}
First it is easy to see that in $S_n$, $\delta(v) = 0$ at the center $v$ and $\delta(u)=2$ for any other vertex $u$. Hence $\Delta(S_n) = 2(n-1)$ as claimed.

On the other hand, for any tree $T$, it is well known that the center $C(T)$ may contain one vertex or two adjacent vertices.
\begin{itemize}
\item If $C(T)$ contains only one vertex, say $v$, then $\delta(v)\geq 0$. And for any vertex $u \neq v$, we have
$$ \ecc(u) \geq \ecc(v)+1 \hbox{ and } \uni(u) \leq \ecc(v)-1. $$
This can be seen by considering a maximal path containing both $v$ and $u$.

Consequently $\delta(u) \geq 2$;

\item If $C(T)$ contains two adjacent vertices $v_1$ and $v_2$, it is easy to see that the edge lies in the middle of a longest path. Then we have
$$\delta(v_1) \geq  1 \hbox{ and } \delta(v_2) \geq 1. $$
For any vertex $u \notin C(T)$, we have $\delta(u) \geq 2$ following similar reasoning as above.
\end{itemize}

In both case we have
$$ \Delta(T) = \sum_{v \in V(T)} \delta (v) \geq 2(n-1). $$
\end{proof}

Now to maximize $\Delta(T)$, like before it may be natural to expect the path to be extremal. The following example shows that this is not the case. We have not yet been able identify the extremal structure in this sense.

Consider the trees $T_1=P_{14}$ and $T_2$ on 14 vertices in Figure~\ref{fig:counterex}, simple computation shows that
$$ \Delta(T_1) = 98 <  104 = \Delta(T_2) . $$

\begin{figure}[htbp]
\centering
    \begin{tikzpicture}[scale=1]

        \node[fill=black,circle,inner sep=1pt] () at (0-8,0) {};
        \node[fill=black,circle,inner sep=1pt] () at (-1-8,0) {};
        \node[fill=black,circle,inner sep=1pt] () at (-2-8,0) {};
        \node[fill=black,circle,inner sep=1pt] () at (1-8,0) {};
        \node[fill=black,circle,inner sep=1pt] () at (2-8,0) {};
        \node[fill=black,circle,inner sep=1pt] () at (3-8,0) {};
        \node[fill=black,circle,inner sep=1pt] (s2) at (4-8,0) {};
        \node[fill=black,circle,inner sep=1pt] (s1) at (-2.5-8,0) {};

\node[fill=black,circle,inner sep=1pt] () at (.5-8,0) {};
\node[fill=black,circle,inner sep=1pt] () at (1.5-8,0) {};
\node[fill=black,circle,inner sep=1pt] () at (2.5-8,0) {};
\node[fill=black,circle,inner sep=1pt] () at (3.5-8,0) {};
\node[fill=black,circle,inner sep=1pt] () at (-.5-8,0) {};
\node[fill=black,circle,inner sep=1pt] () at (-1.5-8,0) {};

\draw (s1)--(s2);

        \node[fill=black,circle,inner sep=1pt] (t1) at (0,0) {};
        \node[fill=black,circle,inner sep=1pt] (t2) at (-1,0) {};
        \node[fill=black,circle,inner sep=1pt] (t3) at (-2,0) {};
        \node[fill=black,circle,inner sep=1pt] (t4) at (1,0) {};
        \node[fill=black,circle,inner sep=1pt] (t5) at (2,0) {};
        \node[fill=black,circle,inner sep=1pt] (t6) at (3,0) {};
        \node[fill=black,circle,inner sep=1pt] (t7) at (4,0) {};
        \node[fill=black,circle,inner sep=1pt] (t8) at (1,-.5) {};

\node[fill=black,circle,inner sep=1pt] () at (.5,0) {};
\node[fill=black,circle,inner sep=1pt] () at (1.5,0) {};
\node[fill=black,circle,inner sep=1pt] () at (2.5,0) {};
\node[fill=black,circle,inner sep=1pt] () at (3.5,0) {};
\node[fill=black,circle,inner sep=1pt] () at (-.5,0) {};
\node[fill=black,circle,inner sep=1pt] () at (-1.5,0) {};

        \draw (t1)--(t2);
        \draw (t2)--(t3);
\draw (t1)--(t4);
        \draw (t4)--(t5);
        \draw (t5)--(t6);

        \draw (t6)--(t7);

        \draw (t4)--(t8);

        \end{tikzpicture}
\caption{The path $T_1$ (on the left) and the tree $T_2$ (on the right) with $\Delta(T_1)  <  \Delta(T_2)$ .}\label{fig:counterex}
\end{figure}
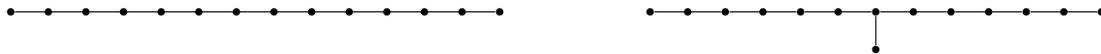

\section{The behavior of $\delta(v)$}
\label{sec:delta}

Last but not least, we treat $\delta(v) = \ecc(v) - \uni(v)$ as a local function and study its properties. Recall that the maximum $\ecc(v)$ is obtained at some leaf and the minimum $\ecc(v)$ is obtained at the center vertices in $C(T)$ (consisting of one or two adjacent vertices). Also recall that $C_{\uni}(T)$ behaves very differently from $C(T)$. It is interesting to see, as we will show in this section, that $\delta(v)$ behaves very much like $\ecc(T)$ in terms of these extremal cases.

First since $\ecc(v)$ is maximized at the end vertices of paths of maximum length, and $\uni(v)=0$ at leaves. The following is trivial.

\begin{prop}\label{th4'}
In a tree $T$ the maximum $\delta(v)$ is obtained at the end vertices of paths of maximum length, exactly those that maximize $\ecc(v)$.
\end{prop}

Next we consider the minimum value of $\delta(v)$ in a tree $T$, denoted by $\delta(T)=\min_{v\in V(T)}\delta(v)$. It turns out that, depending on the parity of the diameter the center vertices either achieves $\delta(T)$ or is very close.

\begin{theorem}\label{theo:ct}
For a tree $T$ with center $C(T)$:
\begin{enumerate}
\item if $C(T)$ consists of a single vertex $v$, then $\delta(v)=\delta(T)$;
\item if $C(T)$ contains two vertices, then $\delta(v) \leq \delta(T)+1$ for any $v \in C(T)$.
\end{enumerate}
\end{theorem}

\begin{proof}
For Case (1), it is easy to see that the center vertex $v=v_{d/2}$ must be in the middle of the longest path $P:=v_0v_1 \ldots v_d$ for some even $d$. Similar to before, let $T_i$ denote the component containing $v_i$ in $T-E(P)$ for $i=1,2,\ldots, d-1$.

Let $\ecc(v) = d(v_0,v)$ and $\uni(v) = d(u,v)$ for some leaf $u$ in $T_{i_0}$. Suppose, without loss of generality, that $i_0 \leq d/2$ (Figure~\ref{fig:p1}).

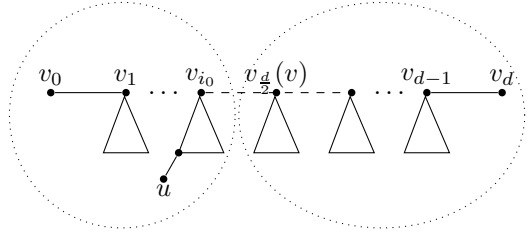
\begin{figure}[htbp]
\centering
    \begin{tikzpicture}[scale=1]
        \node[fill=black,circle,inner sep=1pt] (t1) at (0,0) {};
        \node[fill=black,circle,inner sep=1pt] (t2) at (1,0) {};
        \node[fill=black,circle,inner sep=1pt] (t3) at (2,0) {};
        \node[fill=black,circle,inner sep=1pt] (t8) at (3,0) {};

        \node[fill=black,circle,inner sep=1pt] (t5) at (4,0) {};
        \node[fill=black,circle,inner sep=1pt] (t6) at (5,0) {};
        \node[fill=black,circle,inner sep=1pt] (t7) at (6,0) {};
        \node[fill=black,circle,inner sep=1pt] (t9) at (1.7,-.8) {};
        \node[fill=black,circle,inner sep=1pt] (t10) at (1.5,-1.15) {};
        \draw (t9)--(t10);
        \draw (t1)--(t2);
        \draw [dashed] (t3)--(t5);

        \draw (t6)--(t7);

        \draw (t8)--(2.7, -.8)--(3.3, -.8)--cycle;
        \draw (t3)--(1.7, -.8)--(2.3, -.8)--cycle;
        \draw (t2)--(.7, -.8)--(1.3, -.8)--cycle;
        \draw (t5)--(3.7, -.8)--(4.3, -.8)--cycle;
        \draw (t6)--(4.7, -.8)--(5.3, -.8)--cycle;

         \node at (1.5,-1.3) {$u$};
         \node at (0,.2) {$v_{0}$};
         \node at (2,.2) {$v_{i_0}$};
        \node at (3,.2) {$v_{\frac{d}{2}}(v)$};

        \node at (6,.2) {$v_{d}$};
        \node at (1,.2) {$v_{1}$};
        \node at (5,.2) {$v_{d-1}$};
        \draw [dotted] (4.4,-.3) ellipse (1.9cm and 1.5cm);
        \draw [dotted] (0.95,-.3) ellipse (1.5cm and 1.5cm);


        \node at (1.5,0) {$\ldots$};
        \node at (4.5,0) {$\ldots$};

        \end{tikzpicture}
\caption{The tree $T$ with diameter $d$ (even) and vertices $u$, $v$, $v_{i_0}$.}\label{fig:p1}
\end{figure}

Then $$ \delta(v) = i_0 - d(v_{i_0}, u). $$

On the other hand, for a vertex $w$:
\begin{itemize}
\item if $w$ is in $\cup_{i=i_0+1}^{d-1} V(T_i) \cup \{ v_d \}$ (i.e. to the right of $v_{i_0}$), we have
$$ \delta(w) = \ecc(w) - \uni(w) \geq d(w,v_0) - d(w,u) =  i_0 - d(v_{i_0}, u) = \delta(v) . $$
\item if $w$ is in $\cup_{i=1}^{i_0} V(T_i) \cup \{ v_0 \}$, we have
\begin{align*}
\delta(w)  = \ecc(w) - \uni(w) & \geq d(w,v_d) - d(w,u) \\
 & \geq (d- i_0) - d(v_{i_0}, u) \geq  i_0 - d(v_{i_0}, u) = \delta(v) .
\end{align*}
\end{itemize}

For Case (2), following similar notations we have the center vertices $v_{\frac{d\pm 1}{2}}$ on a longest path  $P:=v_0v_1 \ldots v_d$. We also define $T_i$ accordingly for $i=0, 1, \ldots , d$ ($T_0$ and $T_d$ are single vertex components).

For $v_{\frac{d-1}{2}}$, let $\uni(v_{\frac{d-1}{2}}) = d(v_{\frac{d-1}{2}}, u)$ for some leaf $u$ in $T_{i_0}$.
\begin{itemize}
\item If $i_0 \geq \frac{d+1}{2}$, then exactly the same argument as Case (1) leads to $\delta(v_{\frac{d-1}{2}}) = \delta(T)$.
\item If $i_0 \leq \frac{d-1}{2}$, then
\begin{align*}
 \delta(v_{\frac{d-1}{2}}) & =  \ecc(v_{\frac{d-1}{2}}) - \uni(v_{\frac{d-1}{2}}) = d(v_{\frac{d-1}{2}}, v_d) - d(v_{\frac{d-1}{2}}, u) \\
& = \frac{d+1}{2} - \left( \left( \frac{d-1}{2} - i_0 \right) + d(u, v_{i_0}) \right)  = 1+i_0 - d(u, v_{i_0}).
\end{align*}

For any other vertex $w$ similar arguments as Case (1) shows that $\delta(w) \geq i_0 - d(u, v_{i_0})$. Hence $\delta(v_{\frac{d-1}{2}}) \leq \delta(T)+1$.
\end{itemize}

The argument for $v_{\frac{d+1}{2}}$ is completely the same.
\end{proof}

\begin{remark}
The center vertex or vertices are not necessarily the only ones achieving the minimum $\delta(v)$. As can be easily seen from Figure~\ref{fig.example3}, that $\delta(u)=\delta(v)=\delta(w)=2$ in the tree $T$.
\end{remark}

\section{Concluding remarks}

For any vertex in a tree, when taking the minimum instead of maximum distance to any leaf vertex, we have the {\it uniformity} (as opposed to the eccentricity) at a vertex. Similarly, instead of taking the sum of eccentricities in a tree (as was previously studied), one may take the sum of smallest distance from each vertex to leaves, or take the largest sum of distance from a vertex to others (i.e. the largest value of distance function). These concepts appear to be natural variations of eccentricity and sum of eccentricities. We studied the extremal problems, middle parts of a tree with respect to the new global and local functions. We also compared their behaviors with the eccentricity, center, and sum of eccentricities.

Some of the extremal structures, although natural to expect, does not seem easy to prove. We proposed some related questions along this line.

In addition, in case (2) of Theorem~\ref{theo:ct} we simply claimed that $\delta(v)\leq \delta(T)+1$. But it seems that $\delta(T)$, although not necessarily achieved by both center vertices, can only be achieved at center vertices. Confirming this statement either way would be interesting.

\end {document}